\documentclass[reqno,12pt,letterpaper]{amsart}
\usepackage{amsmath,amssymb,amsthm,graphicx,mathrsfs,url}
\usepackage[usenames,dvipsnames]{color}
\usepackage[colorlinks=true,linkcolor=Red,citecolor=Green]{hyperref}
\usepackage{amsxtra}

\def\?[#1]{\textbf{[#1]}\marginpar{\Large{\textbf{??}}}}

\let\epsilon=\varepsilon 

\newtheorem{thm}{Theorem}
\newtheorem{prop}{Proposition}
\newtheorem{defi}[prop]{Definition}

\newtheorem{lem}[prop]{Lemma}
\newtheorem{corr}[prop]{Corollary}
\newtheorem{rem}{Remark}
\theoremstyle{definition}
\newtheorem{ex}{Example}

\numberwithin{equation}{section}

\let\Im=\Imag

\title{0-TH ORDER PSEUDO-DIFFERENTIAL OPERATOR ON THE CIRCLE}
\author{ZHONGKAI TAO}
\email{tzk320581@berkeley.edu}
\address{Department of Mathematics, University of California,
Berkeley, CA 94720, USA}

\begin{document}

    \begin{abstract}
    \noindent
    In this paper we consider 0-th order pseudodifferential operators on the circle. We show that inside any interval disjoint from critical values of the principal symbol, the spectrum is absolutely continuous with possibly finitely many embedded eigenvalues. We also give an example of embedded eigenvalues.
    \end{abstract}

\maketitle

\section{Introduction}
The study of 0-th order pseudodifferential operators has recently attracted new attention because of connections to fluid mechanics (see the work of Colin de Verdi\`ere-Saint-Raymond \cite{CdVS},\cite{CdV} and Dyatlov-Zworski \cite{DyZw}). In this note we address the special case of the circle. We start with a general result valid for any compact manifold.
\begin{thm}\label{thm1}
Let $M$ be a compact smooth manifold.
Suppose $H\in \Psi^0_{1,0}(M)$ with $\sigma(H)=a\in S^0_{1,0}(T^*M)$. Then $${\rm Spec}_{\rm ess}(H)=\{\lambda|\exists (x_j,\xi_j)\in T^*M, |\xi_j|\to \infty, \mbox{ such that}\\ \lim\limits_{j\to \infty}a(x_j,\xi_j)=\lambda\}$$
\end{thm}
This is a slight generalization of \cite[Theorem 2.1]{CdV} where the operators are assumed to have classical symbol (in particular, $a$ is homogeneous of degree $0$).

We then specialize to the case of the circle $S^1=\mathbb{R}/\mathbb{Z}$ and assume $a$ is homogeneous of degree $0$. With the notation reviewed in Section \ref{pre}, we have
\begin{thm}\label{thm2}
Supose $H\in \Psi^0_{1,0}(S^1):L^2(S^1)\to L^2(S^1)$ is self-adjoint and $\sigma(H)$ is homogeneous of degree $0$. Let $N=\{a(x,\pm 1):\partial_x a(x,\pm 1)=0\}$, $a_\pm={\rm min}a(x,\pm 1 )$, $a^\pm={\rm max}a(x,\pm 1 )$, then \begin{enumerate}
    \item ${\rm Spec}_{\rm sc}(H)\subset N$
    \item $\forall I\Subset \mathbb{R}\setminus N$, $|{\rm Spec}_{\rm pp}(H)\cap I|<\infty$
    \item ${\rm Spec}_{\rm ac}(H)\setminus N=[a_-,a^-]\cup [a_+,a^+]\setminus N$.
\end{enumerate}
This means that inside an interval disjoint from the critical values, the spectrum is absolutely continuous with possibly finitely many embedded eigenvalues.
\end{thm}
We conclude this introduction with an example of an operator with embedded eigenvalues.

\begin{ex}
We claim that $H=a^w(x,D)\in \Psi^0_{1,0}(S^1):L^2(S^1)\to L^2(S^1)$ has an embedded eigenvualue at 0, where $a(x,\xi)=\sin(2\pi x)(1-\chi(\xi))$ and $\chi(x)\in C^\infty_0(\mathbb{R})$ such that $\chi(\xi)=1$ at $\xi=(2k-1)\pi,(2k+1)\pi$. The corresponding eigenfunction is $u_k=e^{2k \pi ix },k\in \mathbb{Z}^n$.
Actually, we compute
\begin{align*}
    b(x,D)u_k(x)&=\frac{1}{(2\pi )^n}\int_{\mathbb{R}^n}b(x,\xi)e^{ix\xi}\mathcal{F}(u_k)(\xi)d\xi\\
    &=\int_{\mathbb{R}^n}b(x,\xi)e^{ix\xi}\delta(\xi-2k\pi )d\xi\\
    &=b(x,2k\pi )u_k(x)
\end{align*}
and
\begin{align*}
    a^w(x,D)u_k(x)&=(e^{\frac{i}{2}\left\langle D_x,D_\xi\right\rangle}a)(x,D)u_k(x)\\
    &=(e^{\frac{i}{2}\left\langle D_x,D_\xi\right\rangle}a)(x,2k\pi )u_k\\
    &=\frac{u_k}{\pi ^n}\int_{\mathbb{R}^n\times\mathbb{R}^n}e^{-2iy\eta}a(x+y,2k\pi +\eta)dyd\eta \\
    &=\frac{u_k}{\pi ^n}\int_{\mathbb{T}^n\times\mathbb{R}^n}\sum_{l\in \mathbb{Z}^n}e^{-2i(y+l)\eta}a(x+y,2k\pi +\eta)dyd\eta\\
    &=u_k(x)\int_{\mathbb{T}^n\times\mathbb{R}^n}\sum_{l\in \mathbb{Z}^n}\delta(\eta+l\pi )e^{-2iy\eta}a(x+y,2k\pi +\eta)dyd\eta\\
    &=u_k(x)\sum_{l\in \mathbb{Z}^n}\int_{\mathbb{T}^n}e^{-2\pi i ly}a(x+y,2k\pi +l\pi )dy\\
    &=u_k(x)\sum_{l\in \mathbb{Z}^n}e^{2\pi i lx}\int_{\mathbb{T}^n}e^{-2\pi i lv}a(v,2k\pi +l\pi )dv
\end{align*}
In the sum the terms with $l\neq 1,-1$ will not appear. But the $l=\pm 1$ terms are $0$ due to our construction.
So we get $a^w(x,D)u_k(x)=0$.
\end{ex}
\begin{ex}
The first example can be used to produce an operator on $\mathbb{T}^2$ which has an embedded eigenvalue. In \cite{DyZw}, it was shown that the operator $H_0=\langle D\rangle^{-1}D_{x_2}+2\sin(2\pi x_1)$ on $\mathbb{T}^2$ has absolutely continuous spectrum near $0$. We can show that after perturbation by a $\Psi^{-\infty}$ self-adjoint operator, it will have an embedded eigenvalue at $0$. Define $\chi(\xi,\eta)\in C^\infty_0(\mathbb{R}^2)$ such that $\chi((2k-1)\pi,0)=\chi((2k+1)\pi,0)=1$. Let $b(x,\xi)=2\sin(2\pi x_1)\chi(\xi)$. Then $u_{k,0}=e^{2k\pi ix_1}$ satisfies
$$(H_0-b^w(x,D))u_{k,0}=2(\sin(2\pi x_1)-b^w(x,D))e^{2k\pi ix_1}=0$$
\end{ex}

\section{Preliminaries}\label{pre}
We include some preliminaries for microlocal analysis in this section. For details, see \cite[Chapter 3,4]{GrSj} or \cite[Chapter 4]{Zw}.
\begin{defi}
Let $X$ be an open set in $\mathbb{R}^n$. The space of symbols $S^m_{\rho,\delta}(X\times \mathbb{R}^N)$ is defined to be the space of all $a\in C^{\infty}(X\times \mathbb{R}^N)$ such that for any compact $K\Subset X$ and $\alpha,\beta\in\mathbb{N}^N$, there is a constant $C=C_{K,\alpha,\beta}(a)$ such that
   $$|\partial^\alpha_x\partial^\beta_\theta a(x,\xi)|\leq C(1+|\xi|)^{m-\rho|\beta|+\delta|\alpha|}, (x,\xi)\in K\times \mathbb{R}^N$$
We also define $\bar{S}^m_{\rho,\delta}(\mathbb{R}^n\times \mathbb{R}^N)$ to be the space of all $a\in C^{\infty}(\mathbb{R}^n\times \mathbb{R}^N)$ such that for any $\alpha,\beta\in\mathbb{N}^N$, there is a constant $C=C_{\alpha,\beta}(a)$ such that
   $$|\partial^\alpha_x\partial^\beta_\theta a(x,\xi)|\leq C(1+|\xi|)^{m-\rho|\beta|+\delta|\alpha|}, (x,\xi)\in \mathbb{R}^n\times \mathbb{R}^N$$
\end{defi}
\begin{defi}
Let $a\in S^m_{\rho,\delta}(T^*\mathbb{R}^n)$, $u\in \mathcal{E}'(\mathbb{R}^n)$, we define $$\operatorname{Op}_t(a)u(x)=\frac{1}{(2\pi )^n}\int_{\mathbb{R}^n}\int_{\mathbb{R}^n}a(tx+(1-t)y,\xi)e^{i(x-y)\xi}u(y)dyd\xi$$
    Define standard quantization $a(x,D)=\operatorname{Op}(a)=\operatorname{Op}_1(a)$, Weyl quantization $a^w(x,D)=\operatorname{Op}_{\frac{1}{2}}(a)$. Elements in the space $\Psi^m_{\rho,\delta}(\mathbb{R}^n):=\{\operatorname{Op}(a),a\in S^m_{\rho,\delta}(T^*\mathbb{R}^n)\}$ are called pseudodifferential operators.
\end{defi}
\begin{rem}
We have a bijective map $$\sigma:\Psi^m_{1,0}(\mathbb{R}^n)/\Psi^{m-1}_{1,0}(\mathbb{R}^n)\to S^m_{1,0}(T^*\mathbb{R}^n)/S^{m-1}_{1,0}(T^*\mathbb{R}^n)$$
$$A\mapsto \sigma(A):=e^{-ix\xi}A(e^{ix\xi})$$
where $\sigma(A)$ is called the principal symbol of $A$.
\end{rem}
\begin{prop}
Let $a\in \bar{S}^m_{1,0}(T^*\mathbb{R}^n)$, then $\operatorname{Op}(a):H^{s}(\mathbb{R}^n)\to H^{s-m}(\mathbb{R}^n)$ is bounded.
\end{prop}
We also have similar definitions for manifolds, see for example \cite[Section 14.2]{Zw}. In particular, we recall
\begin{prop}\label{cpt1}
Let $M$ be a compact manifold, $H\in \Psi^m_{1,0}(M)$ with $m<0$. Then $H:L^2(M)\to L^2(M)$ is compact.
\end{prop}
\section{General result on the essential spectrum}
In this section we would like to study the essential spectrum of a self-adjoint pseudodifferential operator. The basic example of such operator is the Weyl quantization of a real function. We now give
\begin{proof}[Proof of Theorem \ref{thm1}]
Let $$A=\{\lambda|\exists (x_j,\xi_j)\in T^*M, |\xi_j|\to \infty, \mbox{ such that} \lim\limits_{j\to \infty}a(x_j,\xi_j)=\lambda\}$$ First we prove ${\rm Spec}_{\rm ess}(H)\subset A$. If $\lambda\notin A,$ then $\exists \epsilon, N >0 $ such that $|a(x,\xi)-\lambda|>\epsilon, \forall |\xi|>N$. So $\exists b\in C^\infty$ such that $b(x,\xi)=\frac{1}{a(x,\xi)-\lambda}, \forall |\xi|>N$.\\
Then $(\operatorname{Op}(a)-\lambda)\operatorname{Op}(b)=I-\operatorname{Op}(c)$, where $c\in S^{-1}_{1,0}$. By proposition \ref{cpt1}, $c(x, D)$ is compact, so $\operatorname{Op}(a)-\lambda$ has a right inverse up to a compact operator. Similarly, it has a left inverse up to a compact operator. So $\operatorname{Op}(a)-\lambda$ is a Fredholm operator, and hence its spectrum near 0 is discrete. So $\lambda$ is in the discrete spectrum. So ${\rm Spec}_{\rm ess}(H)\subset A$.

 Then we prove $A\subset {\rm Spec}_{\rm ess}(H)$. If $\lambda\in A$, then $\exists (x_0,\xi_j), |\xi_j|\to \infty$ such that $a(x_0,\xi_j)\to\lambda$, So for $0\leq\chi\in C^\infty(M)$, $\int_{M}\chi^2(x)=1 $, we have
$$(\operatorname{Op}(a)-\lambda)(\chi e^{ix\xi_j})=[\operatorname{Op}(a),\chi(x)]e^{ix\xi_j}+(a(x,\xi_j)-\lambda)\chi(x)e^{ix\xi_j}$$
Choose $\epsilon>0$. Since $|a(x,\xi)-a(x_0,\xi)|\leq \|\partial_xa\|_{L^\infty}|x-x_0|$, so we choose $\chi$ such that ${\rm diam}({\rm supp}(\chi))<\frac{\epsilon}{3\|\partial_xa\|_{L^\infty}}$, so that $|a(x_0,\xi)-a(x,\xi)|<\frac{\epsilon}{3}$ on ${\rm supp}\chi$ . Now choose $j$ big enough such that $|a(x_0,\xi_j)-\lambda|<\frac{\epsilon}{3}$, so that $|a(x,\xi_j)-\lambda|< \frac{2\epsilon}{3}$ on ${\rm supp}\chi$ and $$\|(a(x,\xi_j)-\lambda)\chi(x)e^{ix\xi_j}\|_{L^2}<\frac{2\epsilon}{3}$$ Since $[\operatorname{Op}(a),\chi]\in\Psi^{-1}_{1,0}$, it is a compact operator by Proposition \ref{cpt1}. We know $e^{ix\xi_j}$ tends to zero weakly by Riemann Lebesgue lemma, so when $j$ is big enough, we have $$\|[\operatorname{Op}(a),\chi]e^{ix\xi_j}\|_{L^2}<\frac{\epsilon}{3}$$ In conclusion, $\exists N_0$ such that $\forall j>N_0$, we have $\|(\operatorname{Op}(a)-\lambda)\chi e^{ix\xi_j}\|_{L^2}<\epsilon$, where $\|\chi e^{ix\xi_j}\|_{L^2}=1$. So $H-\lambda$ is not invertible, so $\lambda\in {\rm Spec}(H)$.

If $\lambda\in {\rm Spec}(H)\setminus {\rm Spec}_{\rm ess}(H)={\rm Spec}_{\rm d}(H)$, then there exists $\delta>0$ such that ${\rm Spec}(H)\cap(\lambda-\delta,\lambda+\delta)=\{\lambda\}$. Make an orthogonal decomposition $V=L^2(M)=V_1\oplus V_2={\rm ker}(H-\lambda)\oplus {\rm ker}(H-\lambda)^\perp$. Let $P:V\to V_1$ be the corresponding projection and $Q=I-P$. Now $H-\lambda$ is invertible on $H_2$ with $\|(H-\lambda)^{-1}Q\|\leq C=\delta^{-1}$. Select $0\leq \chi_1,...,\chi_N\in C^\infty(M)$ such that ${\rm supp}\chi_k\cap {\rm supp}\chi_l=\emptyset, \forall k\neq l$ and $${\rm diam}({\rm supp}(\chi_k))<\frac{\epsilon}{3\|\partial_xa\|_{L^\infty}}$$ $$\int_{M}\chi_k^2(x)=1,k=1,2,\dots,N$$ Let $u_k=\chi_k e^{ix\xi_j}$ for big enough $j$ such that $\|(\operatorname{Op}(a)-\lambda)u_k\|_{L^2}<\epsilon$. We have $$\|Qu_k\|\leq C\|(H-\lambda)Qu_k\|=C\|(H-\lambda)u_k\|<C\epsilon$$ Then $$\|u_k\|^2=\|Pu_k\|^2+\|Qu_k\|^2\leq\|Pu_k\|^2+C^2\epsilon^2$$ So $\|Pu_k\|^2\geq 1-C^2\epsilon^2$. On the other hand, we have $|\left\langle Pu_k,Pu_l\right\rangle|=|\left\langle Qu_k,Qu_l\right\rangle|\leq C^2\epsilon^2$ for $l\neq k$. Therefore, $\forall l\leq \frac{(1-C^2\epsilon^2)}{C^2\epsilon^2}$, every $l$ elements of $\{Pu_k\}$ form a linearly independent subset of $V_1$. So $${\rm dim} V_1\geq{\rm min}\left\{N,\left[\frac{(1-C^2\epsilon^2)}{C^2\epsilon^2}\right]\right\}$$ Let $N$ big and $\epsilon$ small, then ${\rm dim}\,{\rm ker}\, H$ can be larger than any number. So ${\rm dim}\,{\rm ker}\, H=\infty$. We get a contradiction with $\lambda\in {\rm Spec}_{\rm d}(H)$.
\end{proof}

The essential spectrum has regular structure by the connectedness of the sphere.
\begin{corr}\label{cor1}
Let $M$ be a compact manifold. Let $a\in S^{0}_{1,0}(T^*M)$ with $H=\operatorname{Op}(a)$. If the dimension $n\geq 2$, then $A={\rm Spec}_{\rm ess}(H)$ is a closed interval. If $n=1$, then $A={\rm Spec}_{\rm ess}(H)$ is the union of two closed intervals.
\end{corr}
\begin{proof}
When $n\geq 2$. Let $\alpha,\beta\in A$ and $\alpha<\lambda<\beta$, we need to prove $\lambda\in A$.

For any $k\in \mathbb{N}$, there exists $|\xi_{j_k}|>k$, $|\eta_{j_k}|>k$ and $x_0,x_1\in M$ such that $|a(x_0,\xi_{j_k})-\alpha|<\frac{1}{k}$ and $|a(x_1,\eta_{j_k})-\beta|<\frac{1}{k}$. Now we can find a continuous path $(\kappa_k(t),\gamma_k(t))$ on $T^*M$ connecting $(x_0,\xi_{j_k})$ and $(x_1,\eta_{j_k})$ with $|\gamma_k(t)|>k$. Let $\alpha+\frac{1}{k}<\lambda<\beta-\frac{1}{k}$, then by connectedness we have $a(\kappa_k(t_k),\gamma_k(t_k))=\lambda$ for some $t_k$. So we get $|\gamma_k(t_k)|\to \infty$ and $\lim\limits_{k\to\infty}a(\kappa_k(t_k),\gamma_k(t_k))=\lambda$. So $\lambda\in A$ as we want. So $A$ is a closed interval.

When $n=1$, we have $A=A_+\cup A_-$ where
$$A_\pm=\{\lambda|\exists (x_j,\xi_j)\in T^*M, \xi_j\to \pm\infty, \mbox{ such that} \lim\limits_{j\to \infty}a(x_j,\xi_j)=\lambda\}$$
By the same method, each term in this union is a closed interval.
  \end{proof}
  The result is a little different in dimension one because $S^0=\{\pm 1\}$ is not connected. One can construct a counterexample using the fact that there are exactly two directions in the real line and a symbol can behave differently in two directions.
\begin{corr}
Make the same assumption with Corollary \ref{cor1}. Let $n\geq 2$. If $\lambda\in A$ is not a limit point of the spectrum, then $A=\{\lambda\}$.
\end{corr}
In this case $\lim\limits_{|\xi|\to \infty}\sup\limits_{x\in M}|a(x,\xi)-\lambda|=0$ and $A-\lambda$ is a compact operator. $\lambda$ is a discrete eigenvalue with ${\rm dim}\,{\rm ker}\,(H-\lambda)=\infty$. Actually $H-\lambda$ has finite rank.

\section{Absence of singular continuous spectrum for the circle}
In this section we want to use Mourre estimate to prove the absence of singular spectrum with some additional assumptions on the symbol.

We would like to study the self-adjoint pseudo-differential operators first. We will need the following lemmas.
\begin{lem}\label{lalala}\cite[Chapter 4]{CFKS}

Let $H:V\to V$ be a bounded self-adjoint operator on the Hilbert space $V$. If there exists a self-adjoint operator $A$, a compact operator $K$ and two nonnegative function $\chi,\chi'\in C^\infty_0(\mathbb{R})$ such that $H\in C^2(A)$ and
$$\chi(H)[iA,H]\chi(H)\geq \chi'(H)+ K$$
with ${\rm supp}(\chi')\subset {\rm supp}(\chi)$, then for any interval $[a,b]\subset\{x:\chi(x)>0\}$, $H$ has absolutely continuous spectrum inside $[a,b]$ with possibly finitely many embedded eigenvalues.
\end{lem}
\begin{rem}\cite[Chapter 7]{AMG}
Let $H\in L(V)$ and $A$ be a self-adjoint operator, then we say $H\in C^k(A)$ if $t\mapsto e^{-itA}Se^{itA}$
is strongly of class $C^k$.
\end{rem}

\begin{lem}\label{lem8}\cite[(8.1),(8.2)]{DiSj}
Let $\chi\in C^\infty_0(\mathbb{R})$, then
there exists an alomst analytic function $\tilde{\chi}\in C^\infty_0(\mathbb{C})$ such that
$|\bar{\partial}\tilde{\chi}|\leq C_N |\Im z|^N,\forall N>0$ and
$\tilde{\chi}|_\mathbb{R}=\chi$.
\end{lem}
\begin{prop}\label{lem9}\cite[Theorem 8.1]{DiSj}
Let $H$ be a self-adjoint operaotr on a Hilbert space, $\chi\in C^2_0(\mathbb{R})$, $\tilde{\chi}\in C^1_0(\mathbb{C})$ be an extension of $\chi$ such that $\bar{\partial}\tilde{\chi}=O(|\Im z|)$, then
$$\chi(H)=-\frac{1}{\pi}\int\bar{\partial}\tilde{\chi}(z)(z-H)^{-1}dz$$
\end{prop}
The method of Mourre estimate applies perfectly to the dimension one case if we add a homogeneity assumption to the principal symbol.
\begin{proof}[Proof of Theorem \ref{thm2}] Let $a\in S^0_{1,0}(S^1\times\mathbb{R})$ such that $H=\operatorname{Op}(a)$ is self-adjoint. By our assumption $a$ has the form $a(x,\xi)=a_0(x,\xi)+a_1(x,\xi)$ where $a_0(x,\xi)=a_0(x,\frac{\xi}{|\xi|})$ for $|\xi|\geq 1$, and $a_1\in S^{-1}_{1,0}(S^1\times\mathbb{R})$. Let $[a,b]\subset \mathbb{R}\setminus N$. Since the essential spectrum is known by the previous section, we only need to prove that $H$ has absolutely continuous spectrum with possibly finitely many embedded eigenvalues inside $[a,b]$.

 The point here is that $a_0$ has only two parts (correspond to two directions on $\mathbb{R}$) $a_0(x,1)$ and $a_0(x,-1)$ so that it is easy to construct Mourre estimate. Since $H$ is self-adjoint we can assume $a_0$ is real function. Let $(a,b)\Subset (a',b')\Subset \mathbb{R}\setminus N$ be intervals that exclude critical values of the principal symbol.

We now take $b=\phi(x,\xi)\xi\in S^1_{\rm cl}(S^1\times\mathbb{R})$ with a real function $\phi=\partial_xa_0(x,\xi)$. Let $A=b^w(x,D)$ be the corresponding self-adjoint operator, then $\sigma({[iA,H]})=\phi(x,\xi)\partial_x a_0(x,\xi)=|\partial_xa_0(x,\xi)|^2\geq C> 0$ on $\{x\in S^1: a_0(x,\pm 1)\in (a',b')\}$. Then we select $\chi\in C^\infty_0(\mathbb{R})$ such that ${\rm supp}(\chi)\subset (a',b')$ and $\chi=1$ on $(a,b)$. We would like to prove there exists a compact operator $K$ such that $\chi(H)[iA,H]\chi(H)\geq C\chi^2(H)+K$.


Let $\tilde{\chi}$ be the almost analytic approximation of $\chi$ as in Lemma \ref{lem8}. Since $((z-a_0)^{-1})(x,D)-(z-a(x,D))^{-1}$ is compact, by Proposition \ref{lem9} we have
\begin{align*}
    \chi(H)u&=-\frac{1}{\pi}\int\bar{\partial}\tilde{\chi}(z)(z-a(x,D))^{-1}udz\\
    &=-\frac{1}{\pi}\int\bar{\partial}\tilde{\chi}(z)((z-a_0)^{-1})(x,D)udz\\
    & +\frac{1}{\pi}\int\bar{\partial}\tilde{\chi}(z)(((z-a_0)^{-1})(x,D)-(z-a(x,D))^{-1})udz\\
     &=(\chi\circ a_0)(x,D)u+Ku\\
\end{align*}
for some compact operator $K$.
So $$\sigma({\chi(H)[iA,H]\chi(H)})=(\chi\circ a_0)^2|\partial_x a_0|^2$$ and $\sigma({\chi^2(H)})=(\chi\circ a_0)^2$. Since $|\partial_x a_0(x,\pm 1)|^2\geq C$ on ${\rm supp \chi\circ a_0}$, we get
$$\chi(H)[iA,H]\chi(H)u\geq C\chi^2(H)u+Ku$$
for some compact operator $K$.

Since $[A,H]$ and $[A,[A,H]]$ are 0-th order pseudodifferential operators, we have $H\in C^2(A)$. So we finish the proof by Lemma \ref{lalala}.
\end{proof}

There is a close relationship between self-adjoint operators and unitary operators by Cayley transform. In \cite{FRA} the method of Mourre estimate is applied to unitary operators to get

\begin{lem}{\cite[THeorem 2.7]{FRA}}
Let $U$ be a unitary operator on a Hilbert space $V$. If there exists a self-adjoint operator $A$, a compact operator $K$, an open interval $\Theta\subset S^1$ and $C>0$ such that $U\in C^2(A)$ and
$$E^U(\Theta)U^*[A,U]E^U(\Theta)\geq C E^U(\Theta)+K$$
Then $U$ has continuous spectrum with possibly finitely many embedded eigenvalues in $\Theta$.
\end{lem}
We can apply this lemma to get a theorem for unitary pseudodifferential operators.
\begin{thm}
Assume we have $a\in S^0_{1,0}(S^1\times\mathbb{R})$ such that $U={\rm Op}(a)$ is unitary. Suppose the principal symbol $a_0(x,\xi)=\sigma(U)$ is homogeneous with respect to $\xi$. Let $\Theta \Subset S^1\setminus(\{a_0(x,1):\partial_xa_0(x,1)=0\}\cup\{a_0(x,-1):\partial_xa_0(x,-1)=0\})$  be an open interval, then $H$ has absolutely continuous spectrum with possibly finitely many eigenvalues inside $\Theta$.
\end{thm}
\begin{proof}
  We have $|a_0(x,\frac{\xi}{|\xi|})|^2-1\in S^{-1}_{1,0}(S^1\times\mathbb{R})$ since $U$ is unitary. Let $\xi\to\infty$ we get $|a(x,\pm 1)|^2=1$, so we may assume $a_0(x,\xi)$ takes values on the unit circle $S^1$. By taking derivative, we get $\bar{a}_0\partial_xa_0+a_0\partial_x\bar{a}_0=0$. So there exists a pure imaginary function $\phi(x,\xi)=a_0(x,\xi)\partial_x\bar{a}_0\in S^0_{1,0}(S^1\times\mathbb{R})$. Then $b(x,\xi)=i\phi(x,\xi)\xi\in S^1_{1,0}(S^1\times\mathbb{R})$ is a real function, so $A=b^w(x,D)$ is a self-adjoint operator. Now
$$\sigma({U^*[A,U]})=|a_0|^2|\partial_x a_0|^2=|\partial_x a_0|^2$$
Since $[A,U]$ and $[A,[A,U]]$ are 0-th order pseudodifferential operators, we have $U\in C^2(A)$. Then we can apply the same method as before to finish the proof.
\end{proof}

In \cite[Section 4]{Na} the unitary scattering matrix $S(\lambda)\in \Psi^0_{1,0}(S^1\times\mathbb{R})$ with the principal symbol $$a(x,\xi)=\exp(\frac{i a\pi }{\sqrt{2\lambda}}\sin x\frac{\xi}{\langle\xi\rangle})$$ is considered. It provides a good example of a unitary operator with absolutely continuous spectrum. The method of Mourre estimate used here is the same as that in \cite{Na}.

\noindent\textbf{Acknowledgements.} This note is based on an undergraduate research project supervised by Maciej Zworski in Berkeley in the spring of 2019. I would like to thank him for introducing this topic and a lot of helpful discussions. The research was supported in part by the National Science Foundation grant DMS-1500852.

\end{document}